\documentclass[12pt, reqno]{amsart}
\usepackage[bookmarksnumbered,colorlinks, plainpages]{hyperref}
\hypersetup{colorlinks=true,linkcolor=red, anchorcolor=green, citecolor=cyan, urlcolor=red, filecolor=magenta, pdftoolbar=true}
\usepackage{fixltx2e}

\RequirePackage{tkz-berge}
\usepackage{amsmath, amsthm, amscd, amsfonts, amssymb, graphicx, color, enumerate, xcolor,  mathrsfs, latexsym}

 \definecolor{cupgreen}{rgb}{0,0.498,0.208}
  \definecolor{cupblue}{rgb}{0,0,.5}
  \definecolor{cupred}{rgb}{1,0.04,0}
  \definecolor{cuppink}{rgb}{0.925,0,0.545}
  \definecolor{cupmagenta}{rgb}{0.624,0.161,0.424}
  \definecolor{cupbrown}{rgb}{0.71,0.212,0.133}

  \definecolor{cupgreen}{rgb}{0,0,0}
  \definecolor{cupblue}{rgb}{0,0,0}
  \definecolor{cupred}{rgb}{0,0,0}
  \definecolor{cuppink}{rgb}{0,0,0}
  \definecolor{cupmagenta}{rgb}{0,0,0}
  \definecolor{cupbrown}{rgb}{0,0,0}

\definecolor{TITLE}{rgb}{0,0,0}
\definecolor{AUTHOR1}{rgb}{0.00,0.59,0.00}
\definecolor{AUTHOR2}{rgb}{0.50,0.00,1.00}
\definecolor{SECTION}{rgb}{0.50,0.00,1.00}
\definecolor{FOOTTITLE}{rgb}{0.00,0.50,0.75}
\definecolor{THM}{rgb}{0.8,0,0.1}
\definecolor{SEC}{rgb}{0,0,1}

\makeatletter
\normalsize
\newtheorem{theorem}{{\color{THM} Theorem}}[section]
\theoremstyle{definition}

\newtheorem{Def}[theorem]{{\color{THM}Definition\ }}

\newtheorem{exa}[theorem]{{\color{THM}Example}}

\newtheorem{remark}[theorem]{{\color{THM}Remark}}
\numberwithin{equation}{section}

\def\stirling#1#2{\genfrac{\{}{\}}{0pt}{}{#1}{#2}}
\usepackage{amscd}
\begin{document}
\title{Mixed restricted Stirling numbers}
\author{Somayeh Barati}
\address{Department of Pure Mathematics, Ferdowsi University of Mashhad, Iran.}
\email{Somaye.barati44@yahoo.com}
\author{Be\'{a}ta B\'{e}nyi}
\address{Faculty of Water Sciences, National University of Public Service, Hungary}
\email{beata.benyi@gmail.com}
\author{Abbas Jafarzadeh}
\address{Department of Pure Mathematics, Ferdowsi University of Mashhad, Iran.}
\email{Jafarzadeh@um.ac.ir}
\author{Daniel Yaqubi}
\address{Faculty of Agriculture and Animal Science, University of Torbat-e Jam, Iran .}
\email{daniel\_yaqubi@yahoo.es}
.
\begin{abstract}
In this note we investigate mixed partitions with extra condition on the sizes of the blocks. We give a general formula and the generating function. We consider in more details a special case, determining the generating functions, some recurrences and a connection to $r$-Stirling numbers. To obtain our results, we use pure combinatorial arguments, classical manipulations of generating functions and  to derive the generating functions we apply the symbolic method. 
\end{abstract}

\maketitle

\section{Introduction}
Set partitions are fundamental and well studied combinatorial objects. Partitions of the set $[n]=\{1,2,\cdots,n\}$ into $k$ non-empty unlaballed blocks are enumerated by the Stirling numbers of the second kind, also called sometimes set partition numbers, denoted by ${n\brace k}$. 
Stirling numbers of the second kind are sometimes introduced by the fundamental recurrence relation:
 \[{n\brace k}={n-1\brace k-1}+k{n-1\brace k},\]
with the initial values 
 ${0\brace 0}=1$
 and
${n\brace 0}={0\brace n}=0$.
The explicit formula 
\[{n\brace k}=\frac{1}{k!}\sum_{m=0}^k(-1)^{k-m}{k\choose m}m^n\]
can be found in the introductory combinatorics books (for instance in \cite{Aig}) as a classical example for the inclusion-exclusion principle. The exponential generating function is a nice example for the use of symbolic method \cite{Fla}:
\[\sum_{n=0}^{\infty}\stirling{n}{k}\frac{x^n}{n!}=\frac{(e^{x}-1)^k}{k!}.
\]
Howard \cite{How} proved the following formula:
\[\stirling{n}{k}=\frac{n!}{k!}\sum_{\substack{i_1+i_2+\cdots+i_k=n}}\frac{1}{i_1!i_2!\cdots i_k!}.\]
The Bell numbers $B_n$ count partitions of a set with $n$ distinct elements into disjoint non-empty sets, thus
 \begin{eqnarray}\label{Bell}
 B_n=\sum_{k=1}^n {n\brace k}.
 \end{eqnarray}
The following classical problem was considered in \cite{Dan}:
\begin{itemize}
\item [$\star$]{Consider $b_1+b_2+\ldots+b_n$ balls with $b_1$ balls labeled by $1$, $b_2$ balls labeled by $2$, $\ldots$, $b_n$ balls  labeled by $n$ and $c_1+c_2+\ldots+c_k$ cells with $c_1$ cells labeled by $1$, $c_2$ cells labeled by $2$, $\ldots$, $c_k$ cells labeled by $k$. Evaluate the number of ways to partition the set of these balls into cells of these types.}
\end{itemize}
In the recent paper, we investigate how extra conditions on the sizes of the cells modify this original problem. We recall some notation and necessary definitions.
A \textit{multiset} $A$ with the \textit{multiplicity mapping} $m$ is a collection of not necessarily different objects such that for each $a\in A$ the number $m(a)$ is the multiplicity of the occupants of $a$ in $A$. If $A$ is a multiset, we denote the set of members of $A$ by $S(A)$ and we call it \textit{the background set of} $A$.
Let $A=[n]$ be the background set of a multiset and $m(i)=b_i$ ($i=1,2,\ldots,n$) its multiplicity mapping. We denote the multiset $(A,m)$ by $\mathcal{A}(b_1,\ldots,b_n)$.
 Using this notation, the classical problem $\star$ is the determination of the number of partitions of the multiset $\mathcal{B}=\mathcal{A}(b_1,\ldots,b_n)$ into multiblocks $\mathcal{C}=\mathcal{A}(c_1,\ldots,c_k)$. Following \cite{Dan}, we refer to these numbers as \emph{mixed partition numbers} and let ${\mathcal{B}\brace\mathcal{C}}$ denote them. 
Further, for the number of mixed partitions that may contain empty blocks we use the notation ${\mathcal{B}\brace\mathcal{C}}_0$.
Clearly,
\[{\mathcal{B}\brace\mathcal{C}}_0=\sum_{1\leqslant i\leqslant k, 0\leqslant  j_i\leqslant c_i}{\mathcal{B}\brace\mathcal{J}},\]
where $\mathcal{J}=\mathcal{A}(j_1,\ldots,j_k)$.

Note that if $b_1=b_2=\ldots=b_n=1$ and $c_1=k, c_2=\ldots=c_k=0$ then ${\mathcal{B}\brace\mathcal{C}}={n\brace k}$ and ${\mathcal{B}\brace\mathcal{C}}_0=\sum_{i=1}^k{n\brace i}$. Moreover, if $b_1=b_2=\ldots=b_n=1$ and $c_1=n, c_2=\ldots=c_k=0$ then ${\mathcal{B}\brace\mathcal{C}}_0=B_n$. For $b_1=b_2=\ldots=b_n=1$ and $c_1=1=c_2=\ldots=c_k=1$ we have ${\mathcal{B}\brace\mathcal{C}}=k!{n\brace k}$.

The authors in \cite{Dan} derived some interesting results about the special case $b_1=b_2=\ldots=b_n=1$ and $c_1=r, c_2=\ldots=c_k=1$ where $n,k$ and $r$ are positive integers.  For the number of partitions of this type, ${\mathcal{B}\brace\mathcal{C}}$ was introduced in \cite{Dan} the shorter notation  $S(n,k,r)$, that we also use in this paper. We call these numbers \emph{mixed Stirling numbers of the second kind}. Similarly, we use $B_0(n,r)$ for ${\mathcal{B}\brace\mathcal{C}}_0$ and call them \emph{mixed Bell numbers}.

In Table \ref{list} we give $S(n,k,r)$ for some small values. 
\begin{table}[h]\label{list}
\begin{tabular}{c|ccccc}
$n/k\atop r=2$ & 1&2&3&4&5\\\hline
2 & 1 &&&&\\
3& 3&3&&&\\
4&7 &18&12&&\\
5&15& 75&120&60&\\
6& 31&270&780&900&360
\end{tabular}
\hspace{0.5cm}
\begin{tabular}{c|ccccc}
$n/k\atop r=3$ &1&2&3&4&5\\\hline
3& 1&&&&\\
4&6&4&&&\\
5&25&40&20&&\\
6&90&260&300&120&\\
7&301&1400&2800&2520&840
\end{tabular}
\begin{tabular}{c|ccccc}
$n/r\atop k=2$ & 1&2&3&4&5\\\hline
2&2&&&&\\
3& 6&3&&&\\
4&14&18&4&&\\
5&30&75&40&5&\\
6&62&270&260&75&6
\end{tabular}\hspace{0.5cm}
\begin{tabular}{c|ccccc}
$n/r\atop k=3$ & 1&2&3&4&5\\\hline
3&6&&&&\\
4& 36&12&&&\\
5&150&120&20&&\\
6&540&780&300&30&\\
7&1806&4200&2800&630&42
\end{tabular}
\caption{ Some values of $S(n,k,r)$}
\end{table}
Special settings arise in different combinatorial problems. For instance, $S(k+1,k,2)$ is the number of hamiltonian circuits in a complete simple graph, or the order of the alternating group $A_n$ (A001710 in \cite{Slo}) ; $S(k+2,k,3)$ is the number of decreasing $3$-cycles in the decompositions of permutations as product of disjoint cycles (A001715 in \cite{Slo}); $S(r+2,2,r)$ and $S(r+3,2,r)$ are the Kekul\'{e} numbers of certain benzeonid (A002411 and A108650 in \cite{Slo}); $S(r+3,3,r)$ is the second leg of Pythagorean triangle with smallest side a cube (A083374 in \cite{Slo}).

\section{Mixed restricted partition numbers}
In this section we consider mixed partitions with an upper bound on the sizes of the blocks. Partitions containing blocks of restricted block sizes were investigated by several authors \cite{Tou, Bon, Cho, Com, Kom}. We recall the definition and some relevant results.  For $n,m,k$ positive integers, the number of partitions of the set $[n]$ into $k$ non-empty blocks such that each block contains at most $m$ elements is called the \emph{restricted Stirling numbers of the second kind} and is denoted by ${n \brace k}_{\leq m}$. Komatsu et al. \cite{Kom} derived the following recurrence relation.
\begin{theorem}\cite{Kom}
Let $n,m$ and $k$ be positive integers. The $\stirling{n}{k}_{\leq m}$ is given by
\[\stirling{n}{k}_{\leq m}=\sum_{i=0}^{m-1}{n\choose i}\stirling{n-i}{k-i}_{\leq m}\text{.}\]
with initial conditions $\stirling{0}{0}_{\leq m}=1$ and $\stirling{n}{0}_{\leq m}=0$, for $n\geq 1$.
\end{theorem}
Similarly, the authors in \cite{Mik} defined \textit{restricted Bell numbers} as 
\[B_{n,\leq m}=\sum_{k=0}^n \stirling{n}{k}_{\leq m}\text{.}\]
It is clear, that $B_{n,\leq m}$ enumerates partitions of $n$ elements into blocks, such that each block contains at most $m$ elements. In accordance with our notation, we denote $\sum_{i=1}^k \stirling{n}{i}_{\leq m}$ by ${n\brace k}_{0,\leq m}$.

Now we turn our attention to mixed partitions. 

\begin{Def}
Let $\mathcal{B}=\mathcal{A}(b_1,\ldots,b_n)$ and $\mathcal{C}=\mathcal{A}(c_1,\ldots,c_k)$. For a positive integer $m$, the number of ways to partition $\mathcal{B}$ balls into  $\mathcal{C}$ non-empty cells such that each block contains at most $m$ elements is  called the \emph{mixed restricted partition numbers} and is denoted by ${\mathcal{B}\brace\mathcal{C}}_{\leq m}$.
\end{Def}

First, we give the formula for the case $b_1=\ldots=b_n=1$ and $ c_1,\ldots,c_k\in\mathbb{N}$, based on the results of \cite{Dan}. 
\begin{theorem}\label{Ext1}
Let $b_1=\ldots=b_n=1, c_1,\ldots,c_k\in\mathbb{N}, \mathcal{B}=\mathcal{A}(b_1,\ldots,b_n)$ and $\mathcal{C}=\mathcal{A}(c_1,\ldots,c_k)$. We have
\[{\mathcal{B}\brace\mathcal{C}}_{\leq m}=\sum_{\substack {\ell_1+\cdots+\ell_k=n\\ \ell_1,\ldots,\ell_k\leq m}}
\binom{n}{\ell_1,\ell_2,\ldots, \ell_k}{\ell_1\brace c_1}_{\leq m}{\ell_2\brace c_2}_{\leq m}\cdots{\ell_k\brace c_k}_{\leq m},\]

where $\binom{n}{\ell_1,\ldots,\ell_k}$ is the multinomial coefficient defined by $\frac{n!}{\ell_1!\ldots \ell_k!}$ with $\ell_1+\cdots +\ell_k=n$.
\end{theorem}
\begin{proof}
To constitute a partition of the set $[n]$ into $\mathcal{C}$ such that $c_1,\ldots,c_k\in\mathbb{N}$ and each block has at most $m$ elements, we proceed as follows. First, choose $\ell_1\leq m$ elements in ${n \choose \ell_1}$ ways  (label them $1$) and partition into $c_1$ blocks in ${\ell_1\brace c_1}_{\leq m}$ ways. Next, choose $\ell_2\leq m$ elements out of the remaining $n-\ell_1$ in ${n-\ell_1 \choose \ell_2}$ ways and partition these elements into $c_2$ blocks in ${\ell_1\brace c_1}_{\leq m}$ ways. By continuing the process we obtain the theorem.
\end{proof}
We investigate a special setting for $c_1,\ldots,c_k$ as follows.
\begin{Def}
Let $n,k$ and $r$ be positive integers, $b_1=b_2=\ldots=b_n=1$, and $c_1=r, c_2=\ldots=c_k=1$, respectively. For a positive integer $m$, we denote ${\mathcal{B}\brace\mathcal{C}}_{\leq m}$ by $S_{\leq m}(n,k,r)$ and call these numbers \emph{mixed restricted Stirling numbers of the second kind}. 
\end{Def}
We illustrate this definition by an example.
\begin{exa}
We evaluate $S_{\leq 2}(3,2,2)$. Suppose that the cells are $(~),(~)$ and $[~]$. The partitions of the set $[3]$ such that each cell contains at most $2$ elements are listed below.
\begin{eqnarray*}
&&\hspace{0.1cm}(1)\hspace{0.2cm}(2)\hspace{0.1cm}[3],\hspace{0.3cm}
(1,2)\hspace{0.1cm}(~)\hspace{0.1cm}[3],\hspace{0.3cm}
(1,3)\hspace{0.1cm}(~)\hspace{0.1cm}[2],\hspace{0.3cm}\\
&&(2,3)\hspace{0.1cm}(~)\hspace{0.1cm}[1],\hspace{0.3cm}
\hspace{0.1cm}(1)\hspace{0.2cm}(3)\hspace{0.1cm}[2],\hspace{0.3cm}
\hspace{0.2cm}(2)\hspace{0.3cm}(3)\hspace{0.1cm}[1],\hspace{0.3cm}\\
&&(1)\hspace{0.1cm}(~)\hspace{0.1cm}[2,3],\hspace{0.3cm}
(2)\hspace{0.1cm}(~)\hspace{0.1cm}[1,3],\hspace{0.3cm}
(3)\hspace{0.1cm}(~)\hspace{0.1cm}[1,2],\hspace{0.3cm}
\end{eqnarray*}
Thus $S_{\leq 2}(3,2,2)=9$.
\end{exa}
\begin{remark}
$S(n,k,r)$ counts a pair of partitions $(\mathcal{OP}_{k-1},\mathcal{P}_{r})$, where $\mathcal{OP}$ is an ordered partition into $k-1$  non-empty blocks and $\mathcal{P}$ is a partition into $r$ blocks, such that each element of $[n]$ appear exactly once.
\end{remark}
It is clear that $S_{\leq m}(n,1,r)=\stirling{n}{r}_{\leq m}$ and $S_{\leq m}(n,k,0)=(k-1)!\stirling{n}{k-1}_{\leq m}$, $S_{\leq m}(n,k,1)=k!\stirling{n}{k}_{\leq m}$.
We derive a fundamental recurrence for $S_{\leq m}(n,k,r)$ using the restricted Stirling numbers.  

\begin{theorem}\label{S(n,k,r)}
For positive integers $n$, $k$ and $r$, $S_{\leq m}(n,k,r)$ is given by
\[S_{\leq m}(n,k,r)=\sum_{i=r}^{n}{n\choose i}\stirling{i}{r}_{\leq m}\stirling{n-i}{k-1}_{\leq m}(k-1)!\text{.}\]
\end{theorem}
\begin{proof}
Choose first $i$ elements in $n\choose i$ ways to create the $r$ non-empty cells in $\stirling{n}{r}_{\leq m}$ ways. Next, fill the $k-1$ labeled cells with the remaining $n-i$ elements, such that the each cell contains at most $m$ elements. This can be done in  $(k-1)!\stirling{n-i}{k-1}_{\leq m}$ ways. Note that we should have $k-1\leqslant n-i$.
\end{proof}
We give next another closed form.
\begin{theorem}
For positive integers $n$, $k$ and $r$, we have
\begin{align*}
S_{\leq m}(n,k,r)=(k-1)!\binom{k+r-1}{k-1}\stirling{n}{k+r-1}_{\leq m}.
\end{align*}
\end{theorem}
\begin{proof}
Take a partition of $n$ elements into $k+r-1$ non-empty blocks. Choose now from these $k+r-1$ blocks $k-1$, and permute them, creating this way the $k-1$ blocks labeled by $2,\ldots, k$. The remaining blocks get the label $1$.
\end{proof}
The following theorem gives a recurrence that involves only the mixed restricted Stirling numbers. 
\begin{theorem}\label{S(n,k,r)2}
For positive integers $n,m,k,r$ the following identity holds
\[S_{\leq m}(n,k,r)=\sum_{i=0}^{m-1}{n-1\choose i}\big((k-1)S_{\leq m}(n-i-1,k-1,r)+S_{\leq m}(n-i-1,k,r-1)\big).\]
\end{theorem}
\begin{proof}
Consider the element $n$, it is included in a cell labeled by $1$ (one of the $r$ cells) or in one of the cells labeled by $2,\ldots, k$. Let $i$ denote the number of elements, that join the cell of $n$. Choose these $i$ elements out of the $n-1$ elements in $\binom{n-1}{i}$ ways. If this cell is labeled by $1$, the remaining elements have to be partitioned into $r-1$ cells all labeled by $1$ and further $k-1$ labeled cells. This can be done in $S_{\leq m}(n-i-1,k,r-1)$ ways. On the other hand, if the cell of $n$ has a label different from $1$, choose a label first for this cell in $k-1$ ways. The remaining elements creates a mixed partitions with $c_1=r$ and $k-2$ further labeled cells. The number of such mixed partitions is $S_{\leq m}(n-i-1,k-1,r)$.  
\end{proof}
Similar arguments lead to the next recurrence.
\begin{theorem}\label{S(n,k,r)3}
For positive integers $n,m,k,r$ we have
\begin{eqnarray*}
S_{\leq m}(n,k,r)&=&S_{\leq m}(n-1,k,r-1)+(k-1)S_{\leq m}(n-1,k-1,r)+(k+r-1)S_{\leq m}(n-1,k,r)\\
&-&\big({n-1\choose m}S_{\leq m}(n-m-1,k,r-1)-(k-1){n-1 \choose m}S_{\leq m}(n-m-1,k-1,r)\big).
\end{eqnarray*}
\end{theorem}
\begin{proof}
Consider the element $n$. There are three cases:
\begin{itemize}
\item[] \textbf{Case I.}
Assume that $n$ is singleton and the cell is labeled by $1$. Therefore, the number of mixed partitions of the set $A'=\{1,2,\cdots,n-1\}$ into $k$ blocks such that $r-1$ are labeled by $1$ is $S_{\leq m}(n-1,k,r-1)$. 
\item[] \textbf{Case II.}
Assume now that $n$ is singleton and the cell is not labeled by $1$. Choose one label in $k-1$ ways and partition the remaining elements in $S_{\leq m}(n-1,k-1,r)$ ways.
\item[] \textbf{Case III.}
Assume now that the element $n$ is not a singleton. We put the element $n$ into a cell after partitioning the set $A=\{1,2,\cdots,n-1\}$ into $\mathcal{C}=\mathcal{A}(c_1,\ldots,c_k)$ cells, where $c_1=r, c_2=\ldots=c_k=1$ and each block has size  at most $m$. This can be done in $(k+r-1)S_{\leq m}(n-1,k,r)$ ways. But if we put $n$ into a block of size $m$, we exceed the restriction on the size of the blocks. How many times did we receive "bad" partitions? 
If it is a block labeled by $1$ there are ${n-1\choose m}(S_{\leq m}(n-m-1,k,r-1))$ such partitions and if it is a block labeled by $2,\ldots k$, then there are 
$(k-1){n-1 \choose m}S_{\leq m}(n-m-1,k-1,r)$ such partitions. We need to reduce the sum with these numbers.  
\end{itemize}
\end{proof}
Next we determine a convolution formula.
\begin{theorem}\label{recs}
For positive integers $n,k,m$ and $r\geq s$, we have
\begin{align*}
S(n,k,r)&=\sum_{j=k+r+1-s}^{n-s}\binom{n}{j}\stirling{n-j}{s}S(j,k,r-s),\quad\mbox{and} \\
S_{\leq m}(n,k,r)&=\sum_{j=k+r+1-s}^{n-s}\binom{n}{j}\stirling{n-j}{s}_{\leq m}S_{\leq m}(j,k,r-s).
\end{align*}
\end{theorem}
\begin{proof}
Begin with creating $s$ blocks labeled by $1$. Let $j$ be the number of the remaining elements, that are partitioned in $r-s$ blocks with label $1$ and other $k-1$ distinct non-empty blocks. 
\end{proof}
For the special value $s=1$ the Theorem \ref{recs} gives
\begin{align*}
S(n,k,r)&=\sum_{j=k+r}^{n-1}\binom{n}{j}S(j,k,r-1),\quad\mbox{and} \\
S_{\leq m}(n,k,r)&=\sum_{j=k+r}^{n-1}\binom{n}{j}S_{\leq m}(j,k,r-1),
\end{align*}
while the special setting $s=r$ leads to \ref{S(n,k,r)}.
We obtain, using similar arguments, the following identities.
\begin{theorem}
For positive integers $n,k,m$ and $k> s$, we have
\begin{align*}
S(n,k,r)&=\sum_{j=k+r+1-s}^{n-s}\binom{n}{j}\binom{k-1}{s}s!\stirling{n-j}{s}S(j,k-s,r),\quad \mbox{and} \\
S_{\leq m}(n,k,r)&=\sum_{j=k+r+1-s}^{n-s}\binom{n}{j}\binom{k-1}{s}s!\stirling{n-j}{s}_{\leq m}S_{\leq m}(j,k-s,r).
\end{align*}
\end{theorem}
\begin{proof}
Begin with creating $s$ labeled blocks. Let $j$ be the number of the remaining elements that are partitioned into $r$ blocks with label $1$ and other $k-1-s$ distinct labeled non-empty blocks.  
\end{proof}
Here, the special case $s=1$ gives:
\begin{align*}
S(n,k,r)&=\sum_{j=k+r}^{n-1}\binom{n}{j}(k-1)S(j,k-1,r),\quad \mbox{and} \\
S_{\leq m}(n,k,r)&=\sum_{j=k+r}^{n-1}\binom{n}{j}(k-1)S_{\leq m}(j,k-1,r).
\end{align*}
\begin{theorem}
For positive integers $n,k,m$, the following identities hold
\begin{align}
rS_{\leq m}(n,k,r)&=\sum_{j=1}^{\min(m,n+2-k-r)}\binom{n}{j}S_{\leq m}(n,k,r-1),\label{rS}\\
nS_{\leq m}(n,k,r)&=\sum_{j=1}^{\min(m,n+2-k-r)}j\binom{n}{j}\left[S_{\leq m}(n,k,r-1)+(k-1)S_{\leq m}(n,k-1,r)\right]\label{nS}.
\end{align}
\end{theorem}
\begin{proof}
The left hand side of the identity \eqref{rS} counts the restricted mixed partitions such that one of the blocks labeled by $1$ is colored. Assume that the colored block contains $j$ elements, choose this set in $\binom{n}{j}$ ways and partition the remaining elements into $k-1$ ordered blocks and $r-1$ blocks, which will obtain the label $1$.  

The left hand side of \eqref{nS} counts the restricted mixed partitions with a single colored element. Assume that the colored element is contained in a block of size $j$. There are $i\binom{n}{i}$ ways to choose the set with $j$ elements and mark one of the elements. If it is a block labeled by one, the remaining elements can be partitioned in $S_{\leq m}(n,k,r-1)$ ways, while if it is a block labeled by another number, there are $(k-1)S_{\leq m}(n,k,r-1)$ possibilities, since we need to choose the label for the block also. 
\end{proof}

\section{Mixed associated partition numbers}
Comtet \cite{Com} introduced the \emph{associated Stirling
numbers of the second kind}, $\stirling{n}{k}_{\geq \ell}$, as the number of partitions of the set $[n]$ into $k$ non-empty blocks such that each block contains at least $\ell$ elements. Associated Stirling numbers received a lot of attention recently \cite{Bon, Cha, CHO2, Cul, Kom, Wang}.

For instance, Komatsu et al. \cite{Kom} proved the following basic recurrence relation.
\begin{theorem}\cite{Kom}
Let $n,\ell$ and $k$ be positive integers. Then $\stirling{n}{k}_{\geq \ell}$ given by
\[\stirling{n}{k}_{\geq \ell}=\sum_{i=\ell}^{n-1}{n-1\choose i}\stirling{n-i-1}{k-1}_{\geq \ell}.\]
\end{theorem}
We introduce the analogue mixed partition numbers and derive the relevant results.
\begin{Def}
Let $\mathcal{B}=\mathcal{A}(b_1,\ldots,b_n)$ and $\mathcal{C}=\mathcal{A}(c_1,\ldots,c_k)$.  For positive integer $\ell$, the number of ways to partition $\mathcal{B}$ balls into  $\mathcal{C}$ non-empty cells such that each cell contains at least $\ell$ element is denoted by ${\mathcal{B}\brace\mathcal{C}}_{\geq \ell}$. These numbers are called the \textit{mixed associated partition numbers}.
\end{Def}
First, we give the analogue version of the theorem given in \cite{Kom}.
\begin{theorem}
Let $b_1=\ldots=b_n=1, c_1,\ldots,c_k\in\mathbb{N}, \mathcal{B}=\mathcal{A}(b_1,\ldots,b_n)$ and $\mathcal{C}=\mathcal{A}(c_1,\ldots,c_k)$. Then we have
\[{\mathcal{B}\brace\mathcal{C}}_{\geq \ell}=\sum_{\substack {i_1+\cdots+i_k=n\\ i_1,\ldots,i_k\geq \ell}}
\binom{n}{i_1,i_2,\ldots, i_k}{i_1\brace c_1}_{\geq \ell}{i_2\brace c_2}_{\geq \ell}\cdots{i_k\brace c_k}_{\geq \ell}.\]
\end{theorem} 
\begin{proof}
The modification of the proof of Theorem \ref{Ext1} is straightforward and left to the reader.
\end{proof}
\begin{Def}
For positive integer $\ell$,  let $n,k$ and $r$ be positive integers, $b_1=b_2=\ldots=b_n=1$ and $c_1=r, c_2=\ldots=c_k=1$. Then we denote ${\mathcal{B}\brace\mathcal{C}}_{\geq \ell}$ by $S_{\geq \ell}(n,k,r)$ and call these numbers \textit{mixed associated Stirling numbers of the second kind}. 
\end{Def}
Next, we list the analogue results of the Theorems \ref{S(n,k,r)}, \ref{S(n,k,r)2}, \ref{S(n,k,r)3}. We omit the proofs which are similar to the restricted case. The proof of the Theorem \ref{ass3} contains some specific ideas, hence we give it in detail. 
\begin{theorem}\label{ass1}
Let $A=\{1,2,\cdots,n\}$. For positive integers $k$ and $r$; $S_{\geq \ell}(n,k,r)$ is given by
\[S_{\geq \ell}(n,k,r)=\sum_{i=r}^{n}{n\choose i}\stirling{i}{r}_{\geq \ell}\stirling{n-i}{k-1}_{\geq \ell}(k-1)!\text{.}\]
\end{theorem}
\begin{theorem}\label{ass2}
Let $n,k,r$ and $\ell$ be positive integers. $S_{\geq \ell}(n,k,r)$ is given by
\[S_{\geq \ell}(n,k,r)=\sum_{i=\ell-1}^{n-1}{n-1\choose i}\big(S_{\geq \ell}(n-i-1,k-1,r)+S_{\geq \ell}(n-i-1,k,r-1)\big).\]
\end{theorem}
\begin{theorem}\label{ass3}
Let $n,k,r$ and $\ell$ be positive integers. Then we have
\begin{eqnarray*}
S_{\geq \ell}(n,k,r)&=&(k+r-1)S_{\geq \ell}(n-1,k,r)\\&+&{n-1\choose \ell-1}\big((k-1)S_{\geq \ell}(n-\ell,k-1,r)+S_{\geq \ell}(n-\ell,k,r-1)\big).
\end{eqnarray*}
\end{theorem}
\begin{proof}
Let us consider the element $n$. First, take a mixed partition of $A'=\{1,2,\ldots,n-1\}$ and put $n$ into any of the $k+r-1$ blocks. This can be done in $(k+r-1)S_{\geq \ell}(n-1,k,r)$ ways. But this way we do not obtain the partitions where $n$ is contained in a block of size exactly $\ell$. This block is labeled by $1$, or by another label. If it belongs to one of the $r$ blocks labeled by $1$ which contains exactly $\ell$ elements, we have ${n-1\choose \ell-1}$ ways to choose the remaining $(\ell-1)$ elements and the number of partitions of the $(n-\ell)$ remaining elements into $k$ labeled blocks such that each block has at least  $\ell$ elements and the number of the blocks labeled by $1$ is $r-1$ is equal to $S_{\geq \ell}(n-\ell,k,r-1)$. On the other hand, if $n$ comes into a block labeled by $j$, $2\leq j\leq k$ which contains exactly $\ell$ elements, we have  ${n-1\choose \ell-1}$ ways to choose the other $(\ell-1)$ elements for this block and $(k-1)$ ways to choose the label. The number of partitions of the $(n-\ell)$ remaining elements into $k-1$ blocks is $S_{\geq \ell}(n-\ell,k-1,r)$.
\end{proof}
\section{Generating functions of mixed Stirling numbers of the second kind}
We devote this section to derive the generating functions of mixed Stirling numbers using the symbolic method \cite{Fla}.

First, we determine the generating function $\stirling{\mathcal{B}}{\mathcal{C}}$ for $b_1=b_2=\cdots=b_n=1$ and general $\mathcal{C}$.
\begin{theorem}
The exponential generating function of $\stirling{\mathcal{B}}{\mathcal{C}}$ with $b_1=\ldots=b_n=1, c_1,\ldots,c_k\in\mathbb{N}, \mathcal{B}=\mathcal{A}(b_1,\ldots,b_n)$ and $\mathcal{C}=\mathcal{A}(c_1,\ldots,c_k)$ is given as follows.
\[\sum_{n\geq 0}\stirling{\mathcal{B}}{\mathcal{C}}\frac{x^n}{n!}=\frac{(e^x-1)^{c_1+c_2\cdots +c_k}}{c_1!c_2!\cdots c_k!}.\]
\end{theorem}
\begin{proof}
The construction for such partitions is 
\[\mbox{SET}_{c_1}(\mbox{SET}_{\geq 1}(\mathcal{X}))\times \mbox{SET}_{c_2}(\mbox{SET}_{\geq 1}(\mathcal{X}))\times \cdots \times \mbox{SET}_{c_k}(\mbox{SET}_{\geq 1}(\mathcal{X})),\]
where we used the usual notation of the symbolic method for the atomic class $\mathcal{X}$, the construction of a $k$-element set $\mbox{SET}_{k}$, and the construction of non-empty set (with at least one element) $\mbox{SET}_{\geq 1}$. This construction turns immediately to 
\[\sum_{n\geq 0}\stirling{\mathcal{B}}{\mathcal{C}}\frac{x^n}{n!}=\frac{(e^x-1)^{c_1}}{c_1!}\frac{(e^x-1)^{c_2}}{c_2!}\cdots\frac{(e^x-1)^{c_k}}{c_k!}.\]
After simplification we get the formula in the theorem.
\end{proof}
Hence, the generating function for the special case $S(n,k,r)$ is:
\[
\sum_{n\geq 0}S(n,k,r)\frac{z^n}{n!}=\frac{(e^z-1)^{r+(k-1)}}{r!}.
\]
Flajolet et al. \cite{Fla} presents a theorem concerning the class of set partitions that can be used directly for deriving generating functions for the number of partitions, respectively mixed partitions satisfying different conditions.
\begin{theorem}\cite{Fla}\label{Flajolet}
The class $S^{(A,B)}$ of set partitions with block sizes in $A\subseteq \mathbb{Z}_{\geq 1}$ and with a number of blocks that belongs to $B$ has exponential generating function:

\begin{align*}
S^{(A,B)}(x)=\beta(\alpha(x)),\qquad\text{where}\qquad \alpha(x)=\sum_{a\in A}\frac{x^a}{a!},\quad \beta(x)=\sum_{b\in B}\frac{x^b}{b!}.\tag{2}
\end{align*}
\end{theorem}
Theorem \ref{Flajolet} gives immediately the exponential generating functions of restricted and associated Stirling numbers  \cite{Kom}: 
\begin{eqnarray*}
\sum_{n=k}^{mk}\stirling{n}{k}_{\leq m}\frac{x^n}{n!}&=&\frac{1}{k!}(\sum_{j=1}^m \frac{x^j}{j!})^k,\\
\sum_{n=\ell k}^{\infty}\stirling{n}{k}_{\geq \ell}\frac{x^n}{n!}&=&\frac{1}{k!}(e^x-\sum_{j=0}^{m-1} \frac{x^j}{j!})^k.
\end{eqnarray*}
We use now this theorem to obtain a recurrence relation for the number of partitions of the set $[n]$ into $k$ blocks such that each block have at most $m$ elements and at least $\ell$ elements.
\begin{Def}
Let $\stirling{n}{k}_{\leq m}^{\geq \ell}$ be the number of partitions of the set $A=\{1,2,\cdots,n\}$ into $k$ blocks such that each block have at most $m$ elements and at least $\ell$ elements.
\end{Def}
We show how to use directly Theorem \ref{Flajolet} to get the generating function for $\stirling{n}{k}_{\leq m}^{\geq \ell}$ and then using classical methods to derive a recurrence. 
\begin{theorem}
For positive integers $n,k$ and $\ell\leq m $ we have 
\[\sum_{n=k\ell}^{mk}\stirling{n}{k}_{\leq m}^{\geq \ell}=\frac{1}{k!} \left(\sum_{j=\ell}^m \frac{x^j}{j!}\right)^k.\]
\end{theorem}
According of the Theorem \ref{Flajolet}, in the current situation we are looking for partitions such that the number of blocks is between $m$ and $\ell$. We can write 
\begin{eqnarray*}
A=\{\ell,\ell+1,\ldots,m\}\qquad\text{and}\qquad \alpha(x)=\sum_{j=\ell}^m\frac{x^j}{j!}.
\end{eqnarray*}
 Since the number of blocks is $k$, we have 
\begin{eqnarray*}
B=\{k\} \qquad\text{and}\qquad \beta(x)=\frac{x^k}{k!} .
\end{eqnarray*}
Hence, the generating function is:
\[\beta(\alpha(x))=\frac{1}{k!} \left(\sum_{j=\ell}^m \frac{x^j}{j!}\right)^k.\]
\begin{theorem}\label{Mos}
For positive integers $n,k$ and $\ell\leq m $ the following identity holds 
\[\stirling{n+1}{k}_{\leq m}^{\geq \ell} =\sum_{i=\ell-1}^{m-1} {n\choose i} \stirling{n-i}{k-1}.\]
\end{theorem}

\begin{proof}
By differentiating the generating function of $\stirling{n}{k}_{\leq m}^{\geq \ell}$, we obtain:
\begin{eqnarray*}
 \frac{1}{(k-1)!} \left(\sum_{j=\ell}^m \frac{x^j}{j!}\right)^{k-1}
\sum_{i=\ell-1}^{m-1} \frac{x^i}{i!}.
\end{eqnarray*}
By extracting coefficients, we have:
\begin{eqnarray*}
\stirling{n+1}{k}_{\leq m}^{\geq \ell}& =&
n! [x^n] \sum_{i=\ell-1}^{m-1} \frac{x^i}{i!}
\frac{1}{(k-1)!} \left(\sum_{j=\ell}^m \frac{x^j}{j!}\right)^{k-1}\\
& =& n! \sum_{i=\ell-1}^{m-1} \frac{1}{i!}
[x^{n-i}] \frac{1}{(k-1)!} 
\left(\sum_{j=\ell}^m \frac{x^j}{j!}\right)^{k-1}\\
& =& n! \sum_{i=\ell-1}^{m-1} \frac{1}{i!}
\frac{1}{(n-i)!} \stirling{n-i}{k-1}_{\leq m}^{\geq \ell}\\
& =& \sum_{i=\ell-1}^{m-1} {n\choose i} \stirling{n-i}{k-1}_{\leq m}^{\geq \ell}.
\end{eqnarray*}
\end{proof}
For the sake of completness, we also give explicitely the generating functions for the mixed Stirling numbers of the second kind and associated Stirling numbers of the second kind. 
\begin{align*}
\sum_{n\geq 0}S_{\leq m}(n,k,r)\frac{x^n}{n!}&=\frac{(\sum_{j=1}^{m}\frac{x^j}{j!})^{r+k-1}}{r!},\\
\sum_{n\geq 0}S_{\geq \ell}(n,k,r)\frac{x^n}{n!}&=\frac{(e^x-\sum_{j=0}^{m}\frac{x^j}{j!})^{r+k-1}}{r!}.
\end{align*}


\section{Connection between $r-$Stirling numbers and mixed Stirling numbers}
The $r$-Stirling numbers were introduced by Border \cite{Bor} as the number of partitions of an $n$-element set into $k$ non-empty subsets such that the first $r$ elements are in distinct subsets. A different approach is described in \cite{Car}.

In \cite{Dan} (Theorem 4.1) the  authors present a bijection between $r$-Stirling numbers of the second kind and mixed Stirling numbers. Here, we give the correct version of this theorem.

\begin{theorem}\label{Fix}
For positive integers $n$, $k$ and $r$ with $r\leq k\leq n$, we have
\[{n\brace k}_r=\sum_{i=0}^{k}{r \choose i}S(n-r,i+1, k-r).\]
\end{theorem}
\begin{proof}
To constitute a partition of the set $[n]$ into $k$ blocks such that the first $r$ elements are in distinct blocks, we put $\{1,2,\ldots, r\}$ in $r$ different blocks. Now, we partition the $n-r$ remaining elements into $k-r$ unlabeled blocks and $r$ labeled blocks allowing for the $r$ labeled blocks to be empty. Let $i$ be the number of the labeled blocks of such a partition, that are not empty. There are $S(n-r,i+1,k-r)$ possibilities to create a partition with $i$ non-empty labeled blocks and $\binom{r}{i}$ ways to choose the labels for these blocks. The theorem follows. 
\end{proof}
\section{Suggestions for further studies}
We considered in this note one special setting of mixed partitions. One could investigate other special settings, not only for $\mathcal{C}$, but also for $\mathcal{B}$. It would be also interesting to consider similar generalization of Stirling numbers of the first kind, respectively Lah numbers. Though these investigations on partitions are very general, there are probably several combinatorial problems, where these special mixed partitions arise. It would be very interesting to find such applications. 
\subsection*{Acknowledgements}
The authors would like to thank the anonymous referee for  some  useful comments.

\end{document}